\documentclass{amsart}
\usepackage[utf8]{inputenc}
\usepackage{amsfonts}
\usepackage{amsmath}
\numberwithin{equation}{section}
\usepackage{amssymb}
\usepackage{float}
\usepackage{tikz-cd}
\usepackage{amsthm}
\usepackage{changes}
\usepackage{enumerate}
\usepackage{tikz}
\usepackage{graphicx}
\usepackage{hyperref}
\usepackage[capitalise]{cleveref}
\usepackage{slashed}  
\usepackage{upgreek}
\usepackage[top=3cm,bottom=3cm,left=2.5cm,right=2.5cm,headsep=10pt,letterpaper]{geometry} 
\usepackage{fancyhdr}
\usepackage{mathrsfs}
\usepackage{dsfont}

\usepackage[
    backend=biber,
    style=alphabetic,
    sorting=nyt
]{biblatex}
\numberwithin{equation}{section}
\theoremstyle{plain}

\newtheorem{theorem}{Theorem}[section]
\newtheorem{proposition}[theorem]{Proposition}
\newtheorem{lemma}[theorem]{Lemma}

\theoremstyle{definition}

\usepackage{tikz}\usetikzlibrary{calc}

\theoremstyle{remark}

\theoremstyle{plain}

\newcommand{\RP}{\mathbb{RP}}

\newcommand{\FS}{\mathrm{FS}}

\newcommand{\CP}{\mathbb{CP}}

\newcommand{\Hess}{\operatorname{Hess}}

\title{Strict Log-Concavity of First Dirichlet Eigenfunctions under Symmetry in $\mathbb{CP}^n$}

\author[J.-L. Hsu]{Jia-Lin Hsu}
\address{Department of Mathematics \\ University of California, Santa Barbara \\  CA93106, USA}
\email{jia-lin@ucsb.edu}
\thanks{J.-L. Hsu is partially supported by NSF DMS 2403557.}

\begin{document}

\begin{abstract}
We prove that the positive first Dirichlet eigenfunction of every smooth domain with convex boundary in $\mathbb{CP}^n$ invariant under the standard $U(1)\times U(n-1)$-action is strictly log-concave for $n=2,3,4$.
\end{abstract}

\maketitle

\section{Introduction}

Let $\Omega$ be a relatively compact smooth domain with convex boundary in a Riemannian manifold $(M,g)$, $u>0$ the first Dirichlet eigenfunction on $\Omega$, i.e.
\begin{align*}
    \Delta_{g}u&=-\lambda_1(\Omega)u && \text{in $\Omega$}\\
    u&=0 && \text{on $\partial \Omega$}.
\end{align*}
The question of whether the convexity of $\partial \Omega$ forces $u$ to be log-concave has a long history. Here, log-concavity means that $\operatorname{Hess}\log u\leq 0$ in $\Omega$. 
Brascamp and Lieb~\cite{MR450480} first proved that the first Dirichlet eigenfunction of every bounded convex domain in $\mathbb{R}^n$ is log-concave. Singer, Wong, Yau, and Yau~\cite{MR829055} later gave another proof with continuity method and used this log-concavity to derive a lower bound for the fundamental gap $\lambda_2-\lambda_1$. This was extended to the sphere~\cite{MR909159}. Andrews and Clutterbuck~\cite{MR2784332} subsequently established a sharp modulus of concavity for $\log u$ and used it to prove the Euclidean fundamental gap conjecture; corresponding results on the sphere were later obtained in~\cite{MR4349140,MR3960269,MR4124117}.
Khan, Nguyen, Tuerkoen, and Wei~\cite{MR4870313} extended this framework to modified logarithmic Hessians on Riemannian manifolds. Wei and Xiao~\cite{wei2025superlogconcavityeigenfunctionshoroconvex} used a constant rank method to prove super log-concavity for horoconvex small domains in hyperbolic space. This log-concavity problem also have been studied in various setting, for example, \cite{dai2026logconcavitylevelsethoroconvexityeigenfunction,wei2025superlogconcavityeigenfunctionshoroconvex, MR1003017,MR4502392,xia2026logconcavitydirichleteigenfunctionsmathbbcp2,MR1799430,MR4931465,MR4913941}.

In this note, we study the log-concavity conjecture for first Dirichlet
eigenfunctions on convex domains in $\mathbb{CP}^n$ that are invariant
under suitable isometric group actions. The conjecture was stated by
Dai, Seto, and Wei~\cite{DaiSetoWei2019} and attributed there to
Zhiqin Lu. We consider domains invariant under the following
$U(1)\times U(n-1)$-action on $\mathbb{CP}^n$:
\begin{align*}
(e^{i\theta},A)\cdot[z_0:z_1:z']
    =[z_0:e^{i\theta}z_1:Az'],
    \qquad A\in U(n-1),
\end{align*}
where $z'=(z_2,\ldots,z_n)$. We obtain the following partial affirmative
answer to Lu's conjecture
\begin{theorem}
    Let $n=2,3,$ or $4$. If $\Omega$ is a $U(1)\times U(n-1)$-invariant smooth domain with convex boundary in $\mathbb{CP}^n$, then its first Dirichlet eigenfunction is strictly log-concave.
\end{theorem}
This is the first strict log-concavity result beyond geodesic balls at the sharp convexity threshold. Xia recently proved super-log-concavity in $\mathbb{CP}^2$ under a stronger convexity condition without symmetry assumption~\cite{xia2026logconcavitydirichleteigenfunctionsmathbbcp2}.

The organization of proof is as follows: in Section 2,
following the one-point continuity framework of Khan, Nguyen, Tuerkoen,
and Wei~\cite{MR4870313}, we deform the domain by mean curvature flow
and study $\operatorname{Hess}\log u_t$ along the deformation.
Huisken's convergence theorem~\cite{MR837523} connects the domain to a
Euclidean ball, while uniqueness preserves the symmetry.  Since strict
log-concavity holds at the Euclidean endpoint and near the spatial
boundary, the continuity argument reduces the problem to a
strictification statement: under the symmetry assumption, weak
log-concavity must imply strict log-concavity. Since mean curvature flow makes closed convex hypersurfaces strictly convex instantaneously, a continuity argument extends the result to domains with convex boundary; see Proposition~\ref{reduction to null contact}.

In Section 3,
we use two different mechanisms to exclude null contacts, according to
the orbit type. On singular orbits, the rigidity of the null direction
allows the strict convexity of the boundary to be carried inward by the
mean curvature flow and converted into strict log-concavity along the singular
orbit. On principal orbits, no boundary information is needed: the
flat Lagrangian geometry of the torus orbit, the isotropy group action on the tangent space, together with the
eigenfunction equation, rules out interior null contacts directly. Thus all possible null contacts are ruled out, completing the strictification argument.

\section*{AI Disclosure}
OpenAI's ChatGPT and Anthropic's Claude were used during the development and preparation of this manuscript as supplementary tools for exploring ideas, refining arguments, and improving the presentation. After establishing the fundamental framework under this symmetry assumption and excluding null contact on the singular orbit stratum, we used ChatGPT to help formulate an early version of the maximum-principle argument and to streamline parts of the proof. Any mathematical content obtained with AI assistance was subsequently reviewed, revised, and verified independently by the author. Responsibility for all results and arguments in the paper rests entirely with the author.

\section*{Acknowledgements}
I would like to thank Professor Guofang Wei for her guidance on this project and Yusen Xia for many valuable discussions. I am also grateful to Sung-Tsang Liao and Chun-Kai Lien for helpful discussions on geometric flows with symmetry.

\section{Fubini--Study Metric and the Deformation Argument}
In this section, we review the basic properties of the Fubini--Study metric by viewing it as an invariant metric on a homogeneous space. We then adapt the deformation argument of~\cite{MR4870313}, based on mean curvature flow of hypersurfaces and the induced deformation of metrics on the enclosed domains. Along this deformation, the null contacts can be excluded near the spatial boundary and
near the Euclidean endpoint. In the end, we conclude with the reduction to an
interior null contact.

The complex projective space $\mathbb{CP}^n$ is the space of complex lines in $\mathbb{C}^{n+1}$. The standard action of $U(n+1)$ on $\mathbb{C}^{n+1}$ induces a transitive action on $\mathbb{CP}^n$, with stabilizer $U(1)\times U(n)$ at $[1:0:\cdots:0]$. Hence
\[
\mathbb{CP}^n \simeq U(n+1)/(U(1)\times U(n)).
\]
The Fubini--Study metric is, up to scaling, the unique
$U(n+1)$-invariant Riemannian metric on $\mathbb{CP}^n$;
see~\cite[Chapter~IX, Example~6.3]{kobayashi1996foundations} and \cite[p.55]{MR76393}. The \(U(n+1)\)-action preserves the complex structure \(J\) and is transitive on the holomorphic two-planes. It follows that the holomorphic sectional curvature is constant, and we normalize the Fubini--Study metric so that its holomorphic sectional curvature is \(4\).

We next describe the Riemannian curvature of the Fubini--Study metric.
The unitary reflection across any complex line induces the geodesic
symmetry at the corresponding point of $\mathbb{CP}^n$, so
$\mathbb{CP}^n$ is a Riemannian symmetric space. Since this symmetry
is an isometry with differential $-\operatorname{Id}$, invariance of
the covariant $5$-tensor $\nabla\operatorname{Rm}$ forces
\(
    \nabla\operatorname{Rm}=0.
\) On the other hand, the curvature
formula for K\"ahler manifolds of constant holomorphic sectional
curvature~\cite[Chapter~IX, Proposition~7.3]{kobayashi1996foundations}
gives
\begin{equation}\label{FS curvature formula}
R(A,B)C=
\langle B,C\rangle A-\langle A,C\rangle B+\langle JB,C\rangle JA-\langle JA,C\rangle JB
+2\langle A,JB\rangle JC,
\end{equation}
In particular,
\(
    \sec(X,Y)
    =
    1+3\langle JX,Y\rangle^2
\)
for every orthonormal pair $X,Y$. and
\[
    \operatorname{Ric}(X,X)=2(n+1),
    \qquad
    R_X(Z)
    :=R(Z,X)X=
    Z-\langle X,Z\rangle X
    +3\langle JX,Z\rangle JX,
\]
where we adopt the convention $\operatorname{Rm}(X,Y,Z,W)\equiv \langle R(X,Y)Z,W\rangle:=\langle \nabla_X\nabla_YZ-\nabla_Y\nabla_XZ-\nabla_{[X,Y]}Z,W\rangle$.

We now turn to the deformation argument based on mean curvature flow. For an oriented hypersurface, we take \(\nu\) to be the outward unit normal and use the convention $h(X,Y):=\langle \nabla_X\nu,Y \rangle$ for the second fundamental form. Then the evolution equation for mean curvature flow of hypersurfaces is $\left(\frac{\partial F}{\partial t}\right)^{\perp}=-H\nu$, where $F_{t}$ is the immersion which parametrizes the hypersurface, and the hypersurface is convex means $h\geq 0$.
Since $\mathbb{CP}^n$ has nonnegative sectional curvature and parallel Riemannian curvature tensor, Huisken established the following convergence result for mean curvature flow of connected strictly convex hypersurfaces.
\begin{proposition}[{\cite[Theorem~1.1 and Remark~1(i)]{MR837523}}]
\label{MCF convergence}
Let $\Omega\subset\CP^n$, $n\geq 2$, be a connected smooth domain with strictly convex boundary. Then the maximally extended mean curvature flow
$\{\Sigma_t\}_{0\leq t<T}$ starting from
$\Sigma_0=\partial\Omega$ shrinks to a single round point $p$. That is, there exist smooth positive scaling factor $c_t\to 0$ such that $\frac{1}{c_t}\mathrm{exp}_{p}^{-1}(\Sigma_t)$ converges to the unit sphere $S^{2n-1}\subset T_p\mathbb{CP}^n$ smoothly as $t\to T$.
\end{proposition}
Huisken's work also yields the following extension from strictly convex to convex initial hypersurfaces.
\begin{proposition}\label{MCF stricification}
Let $\Sigma^{2n-1}\subset\mathbb{CP}^n$, $n\geq2$, be a smooth, closed, embedded convex connected hypersurface. Then the mean curvature flow starting from $\Sigma$ exists smoothly for a short time and becomes strictly convex instantaneously. More precisely, $\Sigma_t$ is strictly convex for every $t>0$ in its interval of smooth existence.
\end{proposition}

\begin{proof}
By the standard short-time existence theory, the mean curvature flow starting from $\Sigma$ exists smoothly on some time interval $[0,T)$. We first show that convexity is preserved. Suppose that the second fundamental form $h\geq 0$ at some space-time point $(p,t)$ and has a null eigenvector $e_1$, and extend $e_1$ to an orthonormal eigenframe $\{e_i\}_{i=1}^{2n-1}$ for $h$. Since $\nabla\operatorname{Rm}=0$ and $K_{\mathbb{CP}^n}>0$, the evolution equation in \cite[Theorem~3.4]{MR837523} gives
$$
\left(\left(\frac{\partial}{\partial t}-\Delta\right)h\right)(e_1,e_1)
=
2\sum_{i=2}^{2n-1}
h_{ii}\,\operatorname{Rm}(e_1,e_i,e_i,e_1)
\geq0,
$$
The tensor maximum principle therefore shows that convexity is preserved.

On the other hand, by \cite[Corollary~3.5(i)]{MR837523},
\(
\left(\frac{\partial}{\partial t}-\Delta\right)H\geq0.
\)
Since convexity is preserved, $H\geq0$. Moreover, $H\not\equiv0$ at $t=0$, since otherwise $h\equiv0$ and $\Sigma_t$ would be totally geodesic, whereas $\mathbb{CP}^n$, $n\geq2$, admits no totally geodesic real hypersurface by the Gauss-Codazzi equation and the curvature formula for $\mathbb{CP}^n$. The scalar strong maximum principle therefore implies that $H>0$ for every $t>0$. Consequently, at any null eigenvector $e_1$ of $h$ at positive time, $
\left(\left(\frac{\partial}{\partial t}-\Delta\right)h\right)(e_1,e_1)>0$, which leads to a contradiction by local contact argument. Hence $\Sigma_t$ is strictly convex for every $t>0$.
\end{proof}
This convergence result also implies that \(\partial\Omega\) is connected. Indeed, evolve \(\partial\Omega=\partial\Omega_0\) by mean curvature flow, writing the evolving boundary as \(\partial\Omega_t\), and let \(\Sigma_t\) be a connected component with the earliest singular time \(T\). By the preceding results, for \(t\) close to \(T\), \(\Sigma_t\) is strictly convex and bounds a small ball-like region. Since \(\Omega_t\) has convex boundary, $\Omega_t$ lies inside this region, and hence every other connected component of \(\partial\Omega_t\) must lie inside it. By shifting the flow of \(\Sigma_t\) slightly forward in time and applying the avoidance principle~\cite{white2024avoidance}, any such component would have to become singular strictly before \(T\), contradicting the choice of \(T\). Hence \(\partial\Omega\) is connected, and the preceding two propositions yield a deformation of the initial domain to the Euclidean unit ball.

\begin{proposition}\label{metric deformation}
Let $\Omega\subset\mathbb{CP}^n$, $n\geq2$, be a connected smooth
domain with convex boundary. Then the mean curvature flow induces a Riemannian metrics
$\{g_t\}_{0\leq t\leq T}$ on $\overline{B_1(0)}$ such that 
\begin{itemize}
    \item $g_{t}$ is continuous in $\mathcal{C}^\infty$-topology.
    \item $g_{t}$ is smooth on $\overline{B_1(0)}\times [0,T)$;
    \item for
$t<T$,
\(
(B_1(0),g_t)
\)
is isometric, up to a constant rescaling, to an enclosed domain of
the mean curvature flow
    \item $g_{T}=g_{euc.}$
\end{itemize}
Moreover, $\partial B_1(0)$ is strictly convex with respect to
$g_t$ for every $t>0$.
\begin{proof}
Let $\Omega_t$ denote the region enclosed by $\Sigma_t$. By
Propositions~\ref{MCF convergence} and~\ref{MCF stricification},
the hypersurfaces are strictly convex for every positive time,
shrink to a single point $p$, and satisfy

$$
\overline{\Omega_{t_2}}\subset\Omega_{t_1},
\qquad
0<t_1<t_2<T.
$$
Since the normal speed is positive, the hypersurfaces
$\{\Sigma_t\}_{0<t<T}$ form a smooth foliation of
$\Omega\setminus{p}$. In particular, the enclosed domains may be
smoothly identified with the unit ball in a way depending smoothly
on $t<T$.

Near the extinction time, we choose these identifications in
geodesic normal coordinates centered at $p$ and apply Huisken's
normalized rescaling. The rescaled domains converge smoothly to
$B_1(0)$, while the rescaled ambient metrics converge smoothly to
$g_{\mathrm{Euc}}$. Pulling back these rescaled metrics to
$B_1(0)$ therefore gives a family $\{g_t\}_{0\leq t<T}$ which is
smooth on $\overline{B_1(0)}\times[0,T)$ and converges to
$g_{\mathrm{Euc}}$ in the $\mathcal C^\infty$-topology as
$t\to T$. Hence, setting
\(
g_T:=g_{\mathrm{Euc}}
\)
extends the family continuously to $[0,T]$ in the
$\mathcal C^\infty$-topology. By construction, for every $t<T$,
$(B_1(0),g_t)$ is isometric, up to a constant rescaling, to the
domain enclosed by the mean curvature flow.

Finally, pullback and constant rescaling preserve the sign of the
second fundamental form. Since $\Sigma_t$ is strictly convex for
every $t>0$, the boundary $\partial B_1(0)$ is strictly convex with
respect to $g_t$ for every $t>0$.
\end{proof}
\end{proposition}

Along this metric deformation with strictly convex boundaries, strict log-concavity holds in a uniform boundary collar and near the Euclidean endpoint.
\begin{lemma}[Space-time boundary exclusion]
\label{lem:space-time-boundary}
Let $\{g_t\}_{0\leq t\leq T}$ be a family of Riemannian
metrics on $\overline{B_1(0)}$ which is continuous in the
$\mathcal C^\infty$-topology on $[0,T]$ and smooth on
$\overline{B_1(0)}\times[0,T)$. Suppose that
$\partial B_1(0)$ is strictly convex with respect to $g_t$ for every $0<t\leq T$ and that $g_T=g_{\mathrm{Euc}}$. Then there exists an open
neighborhood $\mathcal U$ of
$$
\bigl(\partial B_1(0)\times(0,T]\bigr)
\cup
\bigl(\overline{B_1(0)}\times\{T\}\bigr)
$$
in $\overline{B_1(0)}\times[0,T]$ such that the first Dirichlet
eigenfunction $u_t$ of $g_t$ is strictly log-concave on
$$
\mathcal U\cap\bigl(B_1(0)\times[0,T]\bigr).
$$
\begin{proof}
Since the first eigenvalue is simple, standard elliptic regularity
and smooth dependence imply that $u_t$ depends smoothly on $(x,t)$ in $\overline{B_1(0)}\times [0,T)$ with normalization $||u_t||_{L^2(B,g_t)}\equiv 1$; see \cite{MR1814364,MR2160744}. Since $g_{t}\to g_{euc.}$ as $t\to T$ in $\mathcal{C}^\infty$-topology, there is $\epsilon>0$ such that the first eigenfunction is strictly log-concave for every $t\geq T-\epsilon$. On the other hand, by the Hopf-boundary lemma and Lemma 3.4 in \cite{MR3960269}, for every $t>0$, there is an open neighborhood $\mathcal{U}_{t}$ of $\partial B_1(0)\times (t,T-\frac{\epsilon}{2}]$ such that 
$u_t$ is strictly log-concave on $\mathcal{U}_{t}$, Therefore, $\mathcal{U}:=\left(\bigcup_{t>0}\mathcal{U}_t\right)\cup \overline{B_1(0)}\times (T-\epsilon,T]$ meets the requirement.
\end{proof}

\end{lemma}

We conclude the section with the following reduction from failure of strict log-concavity to a null contact.
\begin{proposition}\label{reduction to null contact}
Let $\Omega\subset\mathbb{CP}^n$ be a smooth domain with convex boundary, and assume that its first Dirichlet eigenfunction is not strictly log-concave. Then there exists $t_0\in[0,T)$ such that the first Dirichlet eigenfunction $u_{t_0}$ on the domain $\Omega_{t_0}$ from Proposition~\ref{MCF convergence} is weakly log-concave and has a null direction. More precisely,
\(
\operatorname{Hess}\log u_{t_0}\leq 0
\)
on $\Omega_{t_0}$, and there exist $p\in\Omega_{t_0}$ and a nonzero vector $X\in T_p\mathbb{CP}^n$ such that
\[
\operatorname{Hess}\log u_{t_0}(X,X)=0.
\]
\begin{proof}
If the first eigenfunction on $\Omega$ is already weakly
log-concave, then the failure of strict log-concavity gives a null
direction, and we may take $t_0=0$. Otherwise, there is a unit vector $X\in T_p\mathbb{CP}^n$ such that $\operatorname{Hess}\log u_0(X,X)>0$, and it remains true on $\Omega_\tau$ for all sufficiently small $\tau>0$.

Fix such a $\tau$. Since $\partial\Omega_t$ is strictly convex for every
$t>0$, the metric deformation from $\Omega_\tau$ constructed in Proposition~\ref{metric deformation} satisfies the hypotheses of
Lemma~\ref{lem:space-time-boundary}, and hence  possible null
contacts are excluded near the spatial boundary and near the
Euclidean endpoint. Hence, after removing these regions, all possible null contact points lie in a compact subset of space-time. Let $t_0\in[0,T)$ be the smallest parameter such that
\(
\operatorname{Hess}_{g_{t_0}}\log u_{t_0}\leq0
\)
throughout the interior of domain. Such an $t_0$ exists by continuity and the
strict log-concavity at the Euclidean endpoint. Moreover,
$\operatorname{Hess}_{g_{t_0}}\log u_{t_0}$ cannot be negative
definite on the interior of $\Omega_{t_0}$. Otherwise, by compactness and smooth
dependence, the same strict inequality would hold for parameters
slightly smaller than $t_0$, contradicting its minimality. Thus
there is an interior null contact at $t_0$. Since $(B_1(0),g_{t})$ is isometric to $(\Omega_{t},g_{FS})$ up to a scaling, the first Dirichlet eigenfunction $u_{t_)}$ on $\Omega_t$ is weakly log-concave and has a null contact  eigenvector $X$.
\end{proof}
\end{proposition}
This reduction preserves invariant under ambient isometries, by uniqueness of the mean curvature flow and of the positive normalized first Dirichlet eigenfunction.

\section{Exclusion of null contacts}\label{sec:toric-log-concavity}

By Proposition~\ref{reduction to null contact}, it remains to establish
the strictification. We therefore assume that \(u\) is weakly
log-concave and consider a hypothetical interior null contact of
\(H:=\operatorname{Hess}v\), where \(v:=\log u\). In Section~\ref{Tn invariant domain}, we show that the null space at a
contact point is one-dimensional and use the toric symmetry to place
the null direction in the tangent space of the torus orbit~(Lemma~\ref{lem:null-line-orbit}). Since
\(v\) is constant along each orbit, the Hessian in orbit directions can
then be expressed in terms of \(\nabla v\) and the second fundamental
form of the orbit. In Section~\ref{cohomogeneity two domain}, under the
stronger \(U(1)\times U(n-1)\)-symmetry, we compute this second
fundamental form explicitly~(Lemma~\ref{lem:II-principal}) in the null direction and thereby exclude
null contact when \(n=2\). Finally, in
Section~\ref{further analysis in higher dimension}, we refine this
analysis in higher dimensions and establish a gradient estimate
sufficient to exclude null contact for \(n=3,4\). Most of our argument applies to $U(1)\times U(n-1)$-invariant domains, except for inequality~(\eqref{eq:reduced-maximal-eigenvalue}) and hence the exclusion of null contact (Proposition~\ref{prop:principal-exclusion-CP34}).

\subsection{$\mathbb{T}^n$-invariant Domains}\label{Tn invariant domain}

Given a unit null contact vector $X$, it is natural to consider
$(\Delta H)(X,X)$. Extend $X$ locally to a vector field
$\widetilde X$ with $\nabla \widetilde X(p)=0$. Since $H\leq 0$ and
$X\in\ker H_p$, we have $H_p(X,\cdot)=0$. Hence, at $p$,
\[
\Delta\bigl(H(\widetilde X,\widetilde X)\bigr)
=(\Delta H)(X,X).
\]
Since $H(\widetilde X,\widetilde X)\leq0$ near $p$ and vanishes at
$p$, it follows that $(\Delta H)(X,X)\leq0$. Applying the contact
calculation in \cite[proof of Theorem~2.1]{MR4870313} with $b=0$,
and using the curvature formula in Section~2 and
\(
\operatorname{tr}H=-\lambda-|\nabla v|^2,
\)
we obtain
\begin{equation}\label{null contact equality}
0\leq-\frac12(\Delta H)(X,X)
=\operatorname{tr}\left((H+dv\otimes dv)\circ R_X\right)=-\lambda+3H(JX,JX)-dv(X)^2+3dv(JX)^2,
\end{equation}
which leads to the following simple-nullity result.
\begin{lemma}[Simple nullity at a null contact]
\label{lem:simple-nullity of H}
Let $u>0$ solve $-\Delta u=\lambda u$, with $\lambda>0$, on an open
subset of $\CP^n$. If $H\le0$ near $p$ and $\ker H_p\ne\{0\}$,
then $\dim\ker H_p=1$. Moreover, every nonzero $X\in\ker H_p$
satisfies $dv_p(JX)\ne0$.
\begin{proof}
After normalization, if $dv(JX)=0$ for some unit $X\in\ker H_p$, then
\eqref{null contact equality} and $H\leq0$ give
\[
0\leq-\lambda+3H(JX,JX)-dv(X)^2<0,
\]
a contradiction. Thus $X\mapsto dv(JX)$ has trivial kernel on
$\ker H_p$, and hence $\dim\ker H_p=1$.
\end{proof}
\end{lemma}

Simple nullity makes the null direction unique, so every isometric symmetry fixing $p$ preserving \(u\), and hence \(H\), must preserve this line.
Inspired by Pacini's work~\cite{MR1974691} on mean curvature flow of group orbits, we now use the standard toric symmetry to determine its direction:
$$
(e^{i\theta_1},\ldots,e^{i\theta_n})\cdot[z_0:z_1:\cdots:z_n]
=
[z_0:e^{i\theta_1}z_1:\cdots:e^{i\theta_n}z_n],
$$
which acts isometrically since the Fubini-Study metric is $U(n+1)$-invariant. Writing \(\mathcal O_p=\mathbb T^n\cdot p\) for the torus orbit through \(p\), we show that toric invariance already forces the null line to lie in \(T_p\mathcal O_p\).

\begin{lemma}[The null line is tangent to the orbit]
\label{lem:null-line-orbit}
Let $u>0$ be a $\mathbb T^n$-invariant solution of
$-\Delta u=\lambda u$, with $\lambda>0$, on a torus-invariant open
subset of $\CP^n$. If $H\le0$ near $p$, then
$\ker H_p\subset T_p\mathcal O_p$.
\begin{proof}
If $\ker H(p)=\{0\}$, then this lemma holds trivially. Assume $0\neq X\in \ker H(p)$.
Move $p$ into the real locus $\RP^n$ by the torus action. Notice that complex conjugation
$$
\sigma[z_0:\cdots:z_n]=[\overline z_0:\cdots:\overline z_n]
$$
acts as the reflection across $\RP^n$.
Since \(\sigma\) preserves every torus orbit, it preserves \(u\), \(H\), and hence the unique null line
$\mathbb RX$ given by Lemma~\ref{lem:simple-nullity of H}. Hence
$d\sigma_pX=\pm X$. If $d\sigma_pX=X$, then, since $\sigma$ is
anti-holomorphic,
\(
d\sigma_p(JX)=-JX.
\)
Because $v$ is $\sigma$-invariant, this would imply $dv(JX)=0$,
contrary to Lemma~\ref{lem:simple-nullity of H}. Therefore
\(
d\sigma_pX=-X.
\) Thus the null direction is normal to $\mathbb{RP}^n$.

We next use a second reflection to show that the null direction is tangent to the coordinate projective subspace determined by the nonzero coordinates of \(p\). After relabelling coordinates,
write
\(
p=[r_0:\cdots:r_k:0:\cdots:0]
\)
for $r_j\ne0$,
and let $\Sigma\cong\CP^k$ be the coordinate projective subspace
spanned by the first $k+1$ coordinates. When $k<n$, consider the
holomorphic reflection
$$
\rho[z_0:\cdots:z_n]
=
[z_0:\cdots:z_k:-z_{k+1}:\cdots:-z_n].
$$
At $p$, its differential is $+I$ on $T_p\Sigma$ and $-I$ on
$N_p\Sigma$. Again $\rho$ preserves $u$, hence $H$ and its unique
null line. If $d\rho_pX=-X$, then holomorphicity gives
$d\rho_p(JX)=-JX$, and $\rho$-invariance of $v$ again forces
$dv(JX)=0$, a contradiction. Hence
\(
X\in T_p\Sigma.
\)

Thus \(X\) is tangent to \(\Sigma\) and normal to the real locus. Since
\(
\Sigma\cap\RP^n\cong\RP^k,
\)
it follows that \(X\) is normal to \(\RP^k\) inside \(\Sigma\cong\CP^k\). At \(p\), where all the first \(k+1\) homogeneous coordinates are nonzero, this normal space is precisely generated by the phase variations of these coordinates, and hence equals \(T_p\mathcal O_p\). Therefore
\(
X\in T_p\mathcal O_p,
\)
as claimed.
\end{proof}
\end{lemma}

Since \(v\) is constant along each torus orbit, $\nabla v$ is normal to orbits, and hence the Hessian in an orbit direction is expressed in terms of the second fundamental form of the orbit:
\begin{equation}\label{vertical Hessian formula}
H(A,B)=\langle  \nabla_{A}\nabla v,B\rangle=A\langle \nabla v,B \rangle-\langle \nabla v,\nabla_{A}B \rangle=-\langle \nabla v,\mathds{I}(A,B) \rangle
\end{equation}
where $A,B\in T_p\mathcal{O}_p$ and \(\mathds{I}(A,B):=(\nabla_A B)^\perp\). Thus, the comparison with a torus-invariant flowing hypersurface through the torus orbit reduces to identifying the direction of \(\nabla v\), yielding the following exclusion. This idea is partly motivated by Harvie's work on rotationally symmetric inverse mean curvature flow\cite{MR4568175}.
\begin{proposition}[Exclusion on zero- and one-dimensional orbits]
\label{lem:singular-orbit-exclusion}
Let $\Omega\subset\CP^n$, $n\ge2$, be a connected torus-invariant
smooth domain with convex boundary, and let $u>0$ be a torus-invariant
solution of $-\Delta u=\lambda u$ on $\Omega$, with $\lambda>0$.
If $H\le0$ near $p\in\Omega$ and $\dim\mathcal O_p\le1$, then
$H_p<0$.
\begin{proof}
Suppose $p$ is a null contact. Lemma~\ref{lem:null-line-orbit}
excludes a zero-dimensional orbit. If $\dim\mathcal O_p=1$, take
a unit null vector $X\in T_p\mathcal O_p$.
Lemma~\ref{lem:simple-nullity of H} gives $\nabla v(p)\ne0$.

To exploit the vertical Hessian formula, we use the mean curvature flow to produce a strictly convex torus-invariant hypersurface through \(p\).
By Propositions~\ref{MCF stricification} and~\ref{MCF convergence},
the flow from $\partial\Omega$ is strictly convex at positive times
and shrinks to a point $q$. Uniqueness preserves toric symmetry, so
$q$ is fixed by the torus action and hence $p\ne q$. Since the flow is
mean convex, the enclosed domains are nested and shrink to $q$;
therefore $p$ lies on a strictly convex flowing hypersurface
$\Sigma_t$ for some $t>0$. 

Let $\nu$ be the outward unit normal of $\Sigma_t$ at $p$. By the same
reflection argument in the Proof of Lemma~\ref{lem:null-line-orbit}, $\nabla v(p)$ and $\nu$ lie in the same direction. Hence
$\nabla v(p)=a\nu$ for some $a\ne0$. Therefore, by strict convexity and \eqref{vertical Hessian formula},
\[
0=H(X,X)
=-\langle\nabla v,\mathds I_{\mathcal O_p}(X,X)\rangle
=-\langle a\nu,\mathds{I}_{\Sigma_t}(X,X) \rangle \ne0,
\]
a contradiction.
\end{proof}
\textit{Remark.}
The same argument also applies to convex geodesic balls in $\CP^n$ in
every dimension. In this case, symmetry already determines the direction
of $\nabla v$, so the preceding comparison with mean curvature flow
rules out vertical null contacts. Together with
Proposition~\ref{reduction to null contact}, this gives strict vertical
log-concavity without any ODE analysis.
\end{proposition}

    On the principal orbit stratum, the direction of \(\nabla v\) is no longer determined by symmetry alone, so the preceding comparison with mean curvature flow is no longer available. We therefore return to the null-contact inequality~\ref{null contact equality}. Under the assumption of toric invariance, Lemma~\ref{lem:null-line-orbit} gives $dv(X)=0$. Consequently, the inequality can be expressed as follows:
\begin{equation}\label{KNTW inequality in Hess u}
0\leq -\lambda+3H(JX,JX)-dv(X)^2+3dv(JX)^2
=
-\lambda
+3\left(\frac{\Hess u}{u}\right)(JX,JX).
\end{equation}
To access this \(JX\)-Hessian term, we use \(\mathds I_{\mathcal O_p}(X,X)\) as an intermediate geometric vector. The next lemma relates the Hessian of \(u\) along this vector to the ambient curvature, providing the link for the subsequent comparison with \(JX\).

\begin{lemma}[Orbit-curvature identity]
\label{lem:Hessian=Rm}
Let $u>0$ be a $\mathbb T^n$-invariant smooth function on a
torus-invariant open subset of $\CP^n$, and set $v=\log u$ and
$H=\Hess v$. Suppose $H\leq0$ near $p$ and
$X\in T_p\mathcal O_p$ is a unit null vector of $H_p$. Then, for every
$Y\in T_p\CP^n$,
\[
\left(-\frac{\Hess u}{u}\right)(\mathds{I}_{\mathcal{O}_p}(X,X),Y)
=
\operatorname{Rm}(\nabla v,X,X,Y),
\]
where $\mathds{I}_{\mathcal{O}_p}$ denotes the second fundamental form of the orbit.
\begin{proof}
Extend $X$ as the corresponding fundamental vector field of the torus
action. Since the induced metric on $\mathcal O_p$ is translation
invariant, 
the restriction of \(X\) to \(\mathcal O_p\) is parallel with respect to the induced flat metric, so \(\nabla_XX=\mathds I_{\mathcal O_p}(X,X)\). Since $H$ is torus-invariant, $\mathcal L_XH=0$,
and hence
\[
    X(H(Y,X))=(\mathcal{L}_XH)(Y,X)+H([X,Y],X)+H(Y,[X,X])=0
\]
at $p$. Using the Leibniz rule, we obtain
\[
H(Y,\mathds{I}_{\mathcal{O}_p}(X,X))=
H(Y,\nabla_XX)=-(\nabla_XH)(Y,X)
  -H(\nabla_YX,X)=-(\nabla_XH)(Y,X).
\]
Differentiating the local maximum \(H(X,X)=0\) at \(p\) gives
\((\nabla_YH)(X,X)=0\). Applying the Ricci identity for the Hessian,
$(\nabla_XH)(Y,X)-(\nabla_YH)(X,X) =\operatorname{Rm}(X,Y,\nabla v,X),$
we obtain
\[
-H(Y,\mathds{I}_{\mathcal{O}_p}(X,X))
=\operatorname{Rm}(X,Y,\nabla v,X)=\operatorname{Rm}(\nabla v,X,X,Y).
\]
By \eqref{vertical Hessian formula},
$dv(\mathds{I}_{\mathcal{O}_p}(X,X))=0$, and hence
$\Hess u/u=H+dv\otimes dv$ gives the stated equality.
\end{proof}
\end{lemma}

\subsection{$U(1)\times U(n-1)$-invariant Domain}\label{cohomogeneity two domain}

To determine \(\mathds I_{\mathcal O_p}(X,X)\) explicitly, we now impose \(U(1)\times U(n-1)\)-symmetry. Every \(U(1)\times U(n-1)\)-orbit meets the real locus of \(\Sigma:=\{z_3=\cdots=z_n=0\}\cong\CP^2,
\), so it suffices to consider a contact point \(p\in\Sigma\cap\RP^n\cong\RP^2\). At such a representative, the associated torus orbit is two-dimensional precisely on the principal stratum; otherwise its dimension is at most one, and the null contact is already excluded by Proposition~\ref{lem:singular-orbit-exclusion}. We may therefore restrict attention to the principal case, where the residual isotropy \(U(n-2)\) acts on \(N_p\Sigma\).

\begin{lemma}[Orbit second fundamental form]
\label{lem:II-principal}
Let \(n\ge2\), and let \(u>0\) be a \(U(1)\times U(n-1)\)-invariant
smooth function satisfying \(\Delta u=-\lambda u\) on an invariant open
subset of \(\CP^n\). Suppose \(H\le0\) near a
null contact \(p\in \Sigma\cap \mathbb{RP}^n\) on the principal stratum.
Let $X$ be a unit vector spanning $\ker H_p$, and choose $Y$ so that
$\{X,Y\}$ is an orthonormal basis of $T_p\mathcal{O}_p$. Then $dv(JY)\neq 0$, and
\[
\mathds I_{\mathcal O_p}(X,X)
=
\frac{dv(JY)}
{dv(JX)\,(-H(Y,Y))}
\bigl(dv(JY)JX-dv(JX)JY\bigr),
\]
where $dv(JX)$ and $-H(Y,Y)$ are non-zero by Lemma~\ref{lem:simple-nullity of H}.
\begin{proof}
We first reduce both the gradient and the torus-orbit geometry at the contact point to the totally geodesic coordinate subspace \(\Sigma\cong\CP^2\). If $n=2$, then
$\Sigma=\CP^2$. If $n\ge3$, the isotropy $U(n-2)$ fixes $T_p\Sigma$
and acts in the standard way on $N_p\Sigma$, so invariance gives
$\nabla v\in T_p\Sigma$. On the other hand, at the chosen representative,
$\mathcal O_p=\mathbb T^n\cdot p=\mathbb T^2\cdot p\subset\Sigma$, it is a principal two-dimensional
torus orbit in $\Sigma\cong\CP^2$, and hence is flat and Lagrangian
in $\Sigma$. Since $\Sigma$ is totally geodesic, the second
fundamental form of $\mathcal O_p$ in $\Sigma$ agrees with
$\mathds I_{\mathcal O_p}$ in $\CP^n$.

We now translate the null-contact conditions into a system of equations for the second fundamental form of the torus orbit, using its Lagrangian geometry in \(\Sigma\). Since \(v\) is torus-invariant, \(\nabla v\) is orthogonal to \(T_p\mathcal O_p\). As \(\mathcal O_p\) is Lagrangian in \(\Sigma\), \(\{JX,JY\}\) is an orthonormal basis of its normal space in \(\Sigma\), and hence
$$
\nabla v=dv(JX)JX+dv(JY)JY.
$$
On the other hand, since \(X\) is an eigenvector and \(T_p\mathcal O_p\) is two-dimensional, \(\{X,Y\}\) forms an orthonormal eigenbasis of \(H|_{T_p\mathcal O_p}\). Moreover, by \eqref{vertical Hessian formula},
\(
H(A,B)=-\langle \nabla v,\mathds I_{\mathcal O_p}(A,B)\rangle
=C(A,B,J\nabla v),
\) for every $A,B\in T_p\mathcal O_p$,
where \(C(A,B,C):=\langle\mathds I_{\mathcal O_p}(A,B),JC\rangle\) is totally symmetric for Lagrangian submanifolds; see \cite[p.~2]{MR2776467}.
Combining these facts, we obtain
\begin{align*}
dv(JX)\mathds I_{\mathcal O_p}(X,X)
+dv(JY)\mathds I_{\mathcal O_p}(X,Y)&=0,\\
dv(JX)\mathds I_{\mathcal O_p}(X,Y)
+dv(JY)\mathds I_{\mathcal O_p}(Y,Y)
&=(-H(Y,Y))JY.
\end{align*}

If $dv(JY)=0$, the first identity gives
$\mathds I_{\mathcal O_p}(X,X)=0$, whereas
Lemma~\ref{lem:Hessian=Rm} gives
$R_X(\nabla v)=4\,dv(JX)JX\ne0$, a contradiction. Hence
$dv(JY)\ne0$. Therefore,
\begin{equation}\label{second fundamental form equation XX XY}
\mathds I_{\mathcal O_p}(X,Y)
=
-\frac{dv(JX)}{dv(JY)}
\mathds I_{\mathcal O_p}(X,X)
\end{equation}
\[
\mathds I_{\mathcal O_p}(Y,Y)
=
\frac{-H(Y,Y)}{dv(JY)}JY
+
\frac{dv(JX)^2}{dv(JY)^2}
\mathds I_{\mathcal O_p}(X,X).
\]
Since $\mathcal O_p$ is flat and Lagrangian in $\Sigma$, the Gauss-Codazzi
equation gives
\[
0
=
1+\left\langle
\mathds I_{\mathcal O_p}(X,X),
\mathds I_{\mathcal O_p}(Y,Y)
\right\rangle
-\left|\mathds I_{\mathcal O_p}(X,Y)\right|^2 
=
1+\frac{-H(Y,Y)}{dv(JY)}
\left\langle
\mathds I_{\mathcal O_p}(X,X),JY
\right\rangle.
\]
Thus
\(
\left\langle
\mathds I_{\mathcal O_p}(X,X),JY
\right\rangle
=
\frac{dv(JY)}{H(Y,Y)}.
\)
On the other hand, \eqref{vertical Hessian formula} and $H(X,X)=0$
give
$dv(\mathds I_{\mathcal O_p}(X,X))=0$. Since
$\mathds I_{\mathcal O_p}(X,X)\in\operatorname{span}\{JX,JY\}$,
these two identities determine it completely and yield
\[
\mathds I_{\mathcal O_p}(X,X)
=
\frac{dv(JY)}
{dv(JX)\,(-H(Y,Y))}
\bigl(dv(JY)JX-dv(JX)JY\bigr),
\]
as claimed.
\end{proof}
\end{lemma}

By combining the preceding two lemmas, we can explicitly determine the Hessian term $\mathrm{Hess}u/u(JX,JX)$ in \ref{KNTW inequality in Hess u}.
\begin{proposition}[Normal Hessian formula at a principal null contact]
\label{prop:JX-Hessian-principal}
Under the assumptions of Lemma~\ref{lem:II-principal}, let
$\{X,Y\}$ be the orthonormal basis of $T_p\mathcal O_p$ chosen there.
Then
\[
\left(-\frac{\Hess u}{u}\right)(JX,JX)
=
\frac{dv(JX)^2}
{dv(JX)^2+dv(JY)^2}
\left(
\lambda
+\operatorname{tr}_{N_p\Sigma}H
-4H(Y,Y)
\right).
\]
\begin{proof}
Lemma~\ref{lem:Hessian=Rm} and~\ref{lem:II-principal}  give an implicit relation among the Hessian components. The latter expresses
$\mathds I_{\mathcal O_p}(X,X)$ in terms of $JX$ and $JY$, while the former identifies the action of $-\Hess u/u$ on this vector with $R_X(\nabla v)$. Since
$\nabla v=dv(JX)JX+dv(JY)JY$, the curvature formula gives
$R_X(\nabla v)=4dv(JX)JX+dv(JY)JY$. Taking the $JX$- and
$JY$-components therefore yields
\begin{align*}
\frac{dv(JY)}
{dv(JX)(-H(Y,Y))}
\biggl[
dv(JY)
\left(-\frac{\Hess u}{u}\right)(JX,JX)
-
dv(JX)
\left(-\frac{\Hess u}{u}\right)(JX,JY)
\biggr]
&=
4\,dv(JX),
\\
\frac{dv(JY)}
{dv(JX)(-H(Y,Y))}
\biggl[
dv(JY)
\left(-\frac{\Hess u}{u}\right)(JX,JY)
-
dv(JX)
\left(-\frac{\Hess u}{u}\right)(JY,JY)
\biggr]
&=
dv(JY).
\end{align*}

To eliminate the mixed Hessian term, we take a suitable linear
combination of the two component equations above. We multiply
the first equation by $dv(JY)$ and the second equation
by $dv(JX)$. Moving the common coefficient on the left-hand side to
the right-hand side, we obtain
\begin{align*}
dv(JY)^2
\left(-\frac{\Hess u}{u}\right)(JX,JX)
-
dv(JX)dv(JY)
\left(-\frac{\Hess u}{u}\right)(JX,JY)
&=
4\,dv(JX)^2(-H(Y,Y)),
\\
dv(JX)dv(JY)
\left(-\frac{\Hess u}{u}\right)(JX,JY)
-
dv(JX)^2
\left(-\frac{\Hess u}{u}\right)(JY,JY)
&=
dv(JX)^2(-H(Y,Y)).
\end{align*}
Adding these two equations cancels the mixed term and gives
\begin{equation}\label{no mixed term}
dv(JY)^2
\left(-\frac{\Hess u}{u}\right)(JX,JX)
=
dv(JX)^2
\left[
\left(-\frac{\Hess u}{u}\right)(JY,JY)
-5H(Y,Y)
\right].
\end{equation}

It remains to solve
\(\left(-\frac{\Hess u}{u}\right)(JY,JY),
\)
where we use the trace identity coming from the eigenfunction equation. Since
$\{X,Y,JX,JY\}$ is an orthonormal basis of $T_p\Sigma$,
$H(X,X)=0$, and $dv(X)=dv(Y)=0$, we have
\[
\operatorname{tr}_{T_p\Sigma}\left(-\frac{\Hess u}{u}\right)
=
\left(-\frac{\Hess u}{u}\right)(JX,JX)
+
\left(-\frac{\Hess u}{u}\right)(JY,JY)
-H(Y,Y).
\]
On the other hand, $\nabla v\in T_p\Sigma$, so $dv$ vanishes on
$N_p\Sigma$. Hence, using
$\Hess u/u=H+dv\otimes dv$ together with
$-\Delta u=\lambda u$, the left-hand side equals $\lambda+\operatorname{tr}_{N_p\Sigma}H.$ Therefore,
\begin{equation}\label{eq JY}
\left(-\frac{\Hess u}{u}\right)(JY,JY)
=
\lambda+\operatorname{tr}_{N_p\Sigma}H
-
\left(-\frac{\Hess u}{u}\right)(JX,JX)
+
H(Y,Y).
\end{equation}
Combining \ref{no mixed term} and \ref{eq JY} together and 
collecting the $JX$-Hessian terms gives
\[
\left(-\frac{\Hess u}{u}\right)(JX,JX)
=
\frac{dv(JX)^2}
{dv(JX)^2+dv(JY)^2}
\left(
\lambda+\operatorname{tr}_{N_p\Sigma}H
-4H(Y,Y)
\right).
\]
\end{proof}
\end{proposition}
In $\CP^2$, the preceding formula immediately excludes principal null contacts.
\begin{proposition}[Exclusion of principal null contacts in $\CP^2$]
\label{prop:principal-exclusion-CP2}
Let $u>0$ be a $\mathbb T^2$-invariant solution of
$-\Delta u=\lambda u$, with $\lambda>0$, on a torus-invariant
open subset of $\CP^2$. If $H\le0$ near a point $p$ on a principal
torus orbit, then $H_p<0$.

\begin{proof}
Suppose that $p$ is a null contact. Choose an orthonormal basis
$\{X,Y\}$ of $T_p\mathcal O_p$ as in
Lemma~\ref{lem:II-principal}. Since $n=2$, we have
$N_p\Sigma=\{0\}$, and Proposition~\ref{prop:JX-Hessian-principal}
gives
\[
\left(-\frac{\Hess u}{u}\right)(JX,JX)
=
\frac{dv(JX)^2}
{dv(JX)^2+dv(JY)^2}
\bigl(\lambda-4H(Y,Y)\bigr),
\]
where $H(Y,Y)<0$ by Lemma~\ref{lem:simple-nullity of H}. Thus, $
\left(-\frac{\Hess u}{u}\right)(JX,JX)>0$, and hence 
\[
    -\lambda+3\left(\frac{\Hess u}{u}\right)(JX,JX)<0,
\]
which contradicts to (\ref{KNTW inequality in Hess u}). Therefore, there is no contact point $p$, and $H_p<0$.
\end{proof}
\end{proposition} 

\subsection{The exclusion of Null-Contacts when $n=3$ or $4$}\label{further analysis in higher dimension}

The argument in Proposition~\ref{prop:principal-exclusion-CP2} no longer
gives an immediate contradiction when $n>2$, because the additional
normal contribution $\operatorname{tr}_{N_p\Sigma}H<0$ enters the
Hessian formula. We therefore refine the null-contact argument by
tracking the maximal eigenvalue $\mu_1$ of $H$, rather than testing
$H$ along a fixed local extension of the null vector. 
This is essentially the same idea used by Wei--Xiao\cite{wei2025superlogconcavityeigenfunctionshoroconvex}, with the determinant in place of the maximal eigenvalue.
As the following
proposition shows, this introduces a correction term accounting for
the variation of the corresponding eigendirection.



\begin{proposition}[First variation at a null contact]
\label{prop:full-maximal-eigenvalue-equation}
Let $v=\log u$ and $H=\Hess v$, and suppose that $p$ is a null
contact.  Let $\mu_1$ be the largest eigenvalue of $H$ near $p$ and
extend the unit null vector $X$ locally as its eigenvector field.
Choose an orthonormal eigenbasis
$e_1=X,e_2,\ldots,e_{2n}$ of $H_p$, with eigenvalues
$0=\mu_1>\mu_2\geq\cdots\geq\mu_{2n}$.  Then, for every
$V\in T_p\CP^n$,
\[
\nabla_VX
=
\sum_{\alpha=2}^{2n}
\frac{(\nabla_VH)(X,e_\alpha)}{-\mu_\alpha}\,e_\alpha.
\]
Moreover,
\[
\mathcal E_X
:=
\frac12\left(\Delta\mu_1(p)-(\Delta H)(X,X)\right)
=-\sum_{i=1}^{2n}
H(\nabla_{e_i}X,\nabla_{e_i}X)\geq0.
\]
\end{proposition}

\begin{proof}
We first differentiate the eigenvalue.  Since
$H(X,\cdot)=\mu_1\langle X,\cdot\rangle$ and $X$ has unit length,
differentiation in the direction $V$ gives
$d\mu_1(V)=(\nabla_VH)(X,X)+2H(\nabla_VX,X)$.
The last term is
$2\mu_1\langle\nabla_VX,X\rangle=0$, so $d\mu_1=(\nabla H)(X.X)$. Write
$\nabla_VX=\sum_{\alpha=2}^{2n}c_\alpha e_\alpha$ and extend each
$e_\alpha$ parallel at $p$.  Differentiating
$H(X,e_\alpha)=\mu_1\langle X,e_\alpha\rangle$ gives
$0=(\nabla_VH)(X,e_\alpha)+\mu_\alpha c_\alpha$.
This proves the first equality.

We now compute the Laplacian at a null contact. 
Differentiating $d\mu_1=(\nabla H)(X,X)$ once more and tracing yields
\[
\Delta\mu_1(p)
=(\Delta H)(X,X)
+2\sum_{i=1}^{2n}
(\nabla_{e_i}H)(X,\nabla_{e_i}X).
\]

We finally rewrite the remaining third-derivative term.  Since
$\nabla_{e_i}X\perp X$, differentiating
$H(X,\cdot)=\mu_1\langle X,\cdot\rangle$ and evaluating on
$\nabla_{e_i}X$ at the null contact gives
\begin{align*}
0
&=
\bigl(\nabla_{e_i}
  (H(X,\cdot)-\mu_1\langle X,\cdot\rangle)\bigr)
  (\nabla_{e_i}X)\\
&=
(\nabla_{e_i}H)(X,\nabla_{e_i}X)
+H(\nabla_{e_i}X,\nabla_{e_i}X)\\
&\qquad
-e_i(\mu_1)\langle X,\nabla_{e_i}X\rangle
-\mu_1|\nabla_{e_i}X|^2\\
&=
(\nabla_{e_i}H)(X,\nabla_{e_i}X)
+H(\nabla_{e_i}X,\nabla_{e_i}X).
\end{align*}
Here the last equality uses $\mu_1(p)=0$ and
$\langle X,\nabla_{e_i}X\rangle=0$.
Substitution proves the desired formula for $\mathcal{E}_X$
\end{proof}
Recall that $(\Delta H)(X,X)$ has been computed in
\eqref{null contact equality}, \eqref{KNTW inequality in Hess u},
and Proposition~\ref{prop:JX-Hessian-principal}. It remains to determine
the rotation term $\mathcal E_X$, which measures the variation of the
maximal unit eigenvector field $X$. Symmetry rigidly constrains this
variation; in particular, the $U(n-2)$-action on $N_p\Sigma$ gives
$H|_{N_p\Sigma}=-c$ for some $c>0$. We can therefore compute
$\mathcal E_X$ explicitly as follows.
\begin{proposition}[The rotation term under symmetry]
\label{prop:symmetric-rotation-energy}
At a principal null contact $p\in\Sigma$, let $\{X,Y\}$ be the
orthonormal basis defined in Lemma~\ref{lem:II-principal}.  If
$n\geq3$, then
\begin{equation}
\label{eq:E-explicit}
\mathcal E_X
=
\frac{4|\nabla v|^2
\bigl(|\nabla v|^2+3dv(JX)^2\bigr)}
{dv(JY)^2\,(-H(Y,Y))}
+
\frac{2(n-2)dv(JX)^2}{c}.
\end{equation}

\begin{proof}

Since $\mathcal E_X$ is the trace of the symmetric bilinear form
\(
(V,W)\longmapsto -H(\nabla_VX,\nabla_WX),
\)
we may compute it in any orthonormal frame. We choose
\(
\{X,Y,JX,JY,e_5,\ldots,e_{2n}\},
\)
where $e_5,\ldots,e_{2n}$ form an orthonormal basis of
$N_p\Sigma$.

To compute the contribution of $X$ and $Y$, we extend them to invariant vector fields along the flat $\mathbb{T}^2$-orbit. Then
$\nabla_XX=\mathds I_{\mathcal O_p}(X,X)$ and
$\nabla_YX=\mathds I_{\mathcal O_p}(X,Y)$.  By
(\ref{second fundamental form equation XX XY}),
\(
\mathds I_{\mathcal O_p}(X,Y)
=-\frac{dv(JX)}{dv(JY)}\mathds I_{\mathcal O_p}(X,X)
\).
Moreover, Lemma~\ref{lem:Hessian=Rm} and
$dv(\mathds I_{\mathcal O_p}(X,X))=0$, which follows from (\ref{vertical Hessian formula}), give
\[
-H\bigl(\mathds I_{\mathcal O_p}(X,X),
\mathds I_{\mathcal O_p}(X,X)\bigr)
=
\left\langle R_X(\nabla v),
\mathds I_{\mathcal O_p}(X,X)\right\rangle.
\]
Since
$\nabla v=dv(JX)JX+dv(JY)JY$ and
$R_X(\nabla v)=4dv(JX)JX+dv(JY)JY$,
Lemma~\ref{lem:II-principal} yields
\[
\left\langle R_X(\nabla v),
\mathds I_{\mathcal O_p}(X,X)\right\rangle
=3dv(JY)^2/(-H(Y,Y)).
\]  
Therefore
\(
-H(\nabla_XX,\nabla_XX)-H(\nabla_YX,\nabla_YX)
=
\frac{3|\nabla v|^2}{-H(Y,Y)}.
\)

We next compute the contribution of the $JX$--$JY$ directions.
Complex conjugation acts as $-1$ on $X$.  Differentiating this
equivariance shows that $\nabla_VX\in T_p\mathcal O_p$ whenever
$V\in JT_p\mathcal O_p$.  Since $X$ has unit length,
$\nabla_VX\perp X$, and hence $\nabla_VX\in\mathbb RY$.  By Proposition~\ref{prop:full-maximal-eigenvalue-equation} and the fact that $Y$ is an eigenvector of
$H$, we have
\[
\nabla_VX
=
\frac{(\nabla_VH)(X,Y)}{-H(Y,Y)}\,Y,
\qquad V\in\operatorname{span}\{JX,JY\}.
\]
It is therefore enough to compute $(\nabla_{JX}H)(X,Y)$ and $(\nabla_{JY}H)(X,Y)$.
By the Ricci identity, we have
\[
(\nabla_VH)(X,Y)
=(\nabla_YH)(X,V)+\operatorname{Rm}(V,Y,\nabla v,X).
\]
Because $Y$ is tangent to the torus orbit and $H$ is torus-invariant,
$\mathcal L_YH=0$.  Using $H(X,\cdot)=0$ and
$\nabla_XY=\mathds I_{\mathcal O_p}(X,Y)$, we obtain
\(
(\nabla_YH)(X,V)
=-H(\mathds I_{\mathcal O_p}(X,Y),V)
\) for every $V\in JT_p\mathcal{O}_p$.
Using
\(
\mathds I_{\mathcal O_p}(X,Y)
=-\frac{dv(JX)}{dv(JY)}\mathds I_{\mathcal O_p}(X,X)
\), Lemma~\ref{lem:Hessian=Rm} and~\ref{lem:II-principal} again, we have
\(
-H(\mathds I_{\mathcal O_p}(X,Y),JX)
=-\frac{4dv(JX)^2}{dv(JY)}\),
and
\(
-H(\mathds I_{\mathcal O_p}(X,Y),JY)
=-dv(JX).
\)
The curvature formula of $\CP^n$ also gives
\(
\operatorname{Rm}(JX,Y,\nabla v,X)=-dv(JY)
\) and
\(
\operatorname{Rm}(JY,Y,\nabla v,X)=-2dv(JX)
\).  Hence
\[
(\nabla_{JX}H)(X,Y)
=-\frac{dv(JY)^2+4dv(JX)^2}{dv(JY)},
\qquad
(\nabla_{JY}H)(X,Y)=-3dv(JX).
\]
Substitution into the formula in Proposition~\ref{prop:full-maximal-eigenvalue-equation} . Then we have
\[
-H(\nabla_{JX}X,\nabla_{JX}X)
-H(\nabla_{JY}X,\nabla_{JY}X)
=
\frac{|\nabla v|^2
\bigl(|\nabla v|^2+15dv(JX)^2\bigr)}
{dv(JY)^2\,(-H(Y,Y))}.
\]

We finally compute the contribution of the isotropy-normal
directions. Notice that the map
$V\mapsto(\nabla_VH)(X,\cdot)$ is $U(n-2)$-equivariant.  Since
$U(n-2)$ acts trivially on $T_p\Sigma$ and by its standard unitary
representation on $N_p\Sigma$, its $T_p^*\Sigma$-component vanishes,
and its normal component is a linear combination of
$\langle V,\cdot\rangle$ and $\langle JV,\cdot\rangle$.  It is thus
enough, for a unit $W\in N_p\Sigma$, to evaluate this covector on
$W$ and $JW$.

We compute these two values using the Ricci identity.  Extend $W$
parallelly. Then $W$ lie in the normal bundle along the $\Sigma$ since it is totally geodesic, and 
$(\nabla_XH)(W,W)=(\nabla_XH)(W,JW)=0$ because
$H|_{N_p\Sigma}=-c\,g_{\FS}$ is constant, Therefore,
\[
(\nabla_WH)(X,W)=\operatorname{Rm}(W,X,\nabla v,W)=0,
\qquad
(\nabla_WH)(X,JW)=\operatorname{Rm}(W,X,\nabla v,JW)=dv(JX),
\]
where the last equalities follow from the curvature formula (\ref{FS curvature formula}) of
$\CP^n$.  Thus
$(\nabla_WH)(X,\cdot)=dv(JX)\langle JW,\cdot\rangle$.
Proposition~\ref{prop:full-maximal-eigenvalue-equation} and
$H|_{N_p\Sigma}=-c\,g_{\FS}$ now give
$\nabla_WX=dv(JX)JW/c$.  Each complex normal line contributes
$2dv(JX)^2/c$, so
\(
-\sum_{i=5}^{2n}H(\nabla_{e_i}X,\nabla_{e_i}X)
=
\frac{2(n-2)dv(JX)^2}{c}.
\)

We conclude by combining the three orthogonal blocks computed above.
Since
$\mathcal E_X=-\sum_iH(\nabla_{e_i}X,\nabla_{e_i}X)$, we obtain
\begin{align*}
\mathcal E_X
={}&
\frac{3|\nabla v|^2}{-H(Y,Y)}
+
\frac{|\nabla v|^2
\bigl(|\nabla v|^2+15dv(JX)^2\bigr)}
{dv(JY)^2\,(-H(Y,Y))}
+
\frac{2(n-2)dv(JX)^2}{c}
\\
={}&
\frac{|\nabla v|^2
\bigl(3dv(JY)^2+|\nabla v|^2+15dv(JX)^2\bigr)}
{dv(JY)^2\,(-H(Y,Y))}
+
\frac{2(n-2)dv(JX)^2}{c}
\\
={}&
\frac{4|\nabla v|^2
\bigl(|\nabla v|^2+3dv(JX)^2\bigr)}
{dv(JY)^2\,(-H(Y,Y))}
+
\frac{2(n-2)dv(JX)^2}{c}.
\end{align*}
\end{proof}
\end{proposition}

Combining \eqref{null contact equality},
Proposition~\ref{prop:symmetric-rotation-energy},
Proposition~\ref{prop:full-maximal-eigenvalue-equation}, and Proposition~\ref{prop:JX-Hessian-principal}, we obtain
\begin{equation}\label{precise formula for laplace of eigenvalue}
\begin{aligned}
-\frac12\Delta\mu_1(p)
={}&
-\lambda\left(
1+\frac{3dv(JX)^2}{|\nabla v|^2}
\right)
+\frac{12dv(JX)^2H(Y,Y)}{|\nabla v|^2}
+\frac{6(n-2)dv(JX)^2c}{|\nabla v|^2}
\\
&-
\frac{4|\nabla v|^2
\bigl(|\nabla v|^2+3dv(JX)^2\bigr)}
{dv(JY)^2\,(-H(Y,Y))}
-\frac{2(n-2)dv(JX)^2}{c},
\end{aligned}
\end{equation}
where we aim to show is strictly negative to get a contradiction. For \(n\in\{3,4\}\), discarding the
nonpositive last term and using \(6(n-2)\leq12\), we obtain
\begin{equation}
\label{eq:reduced-maximal-eigenvalue}
\begin{aligned}
-\frac12\Delta\mu_1(p)
\leq{}&
-\lambda\left(
1+\frac{3dv(JX)^2}{|\nabla v|^2}
\right)
\\
&-
\frac{4|\nabla v|^2
\bigl(|\nabla v|^2+3dv(JX)^2\bigr)}
{dv(JY)^2\,(-H(Y,Y))}
+
\frac{12dv(JX)^2\bigl(H(Y,Y)+c\bigr)}
{|\nabla v|^2}.
\end{aligned}
\end{equation}
Thus the only potentially positive contribution in
\eqref{eq:reduced-maximal-eigenvalue} is the last term, and we are going to control it by the following gradient estimate.

\begin{lemma}[Principal-orbit gradient estimate]
\label{lem:principal-gradient-estimate}
Let \(n\geq3\), and let \(u>0\) be a
\(U(1)\times U(n-1)\)-invariant solution of
\(-\Delta u=\lambda u\), with \(\lambda>0\), on an invariant open
subset of \(\CP^n\).  Set \(v=\log u\) and \(H=\Hess v\).
Suppose that \(H\leq0\) near a null contact
\(p\in\Sigma\cap\RP^n\) on the principal stratum, and let
\(\{X,Y\}\) be as in Lemma~\ref{lem:II-principal}.  If
\(H|_{N_p\Sigma}=-c\,g_{\FS}\), then
\[
\label{eq:principal-gradient-estimate}
|\nabla v|^2
\geq
(-H(Y,Y))\bigl(H(Y,Y)+c\bigr).
\]
\end{lemma}

\begin{proof}
First, we show that it suffices to find a unit vector $\nu$ such that $dv(\nu)^2=c(c+H(J\nu,J\nu))$. If so, 
\[
\label{eq:radial-gradient-identity}
|\nabla v|^2\geq dv(\nu)^2
=
c\bigl(c+H(J\nu,J\nu)\bigr)\geq c(c+H(Y,Y)).
\]
since \(J\nu\in T_p\mathcal O_p\) is a unit vector and
\(\{X,Y\}\) is an orthonormal eigenbasis of
\(H|_{T_p\mathcal O_p}\), with eigenvalues \(0\) and \(H(Y,Y)\). If
\(H(Y,Y)+c\leq0\), the right-hand side of
\eqref{eq:principal-gradient-estimate} is nonpositive, and this inequality holds trivially since $\nabla v\neq 0$ by Lemma~\ref{lem:simple-nullity of H}.
If \(H(Y,Y)+c>0\), then \(c>-H(Y,Y)\), and this gradient estimate follows.

Now, we construct this $\nu$ by the $U(n-1)$-orbit. Write $p=[a_0:a_1:a_2:0:\cdots:0]$.
Let \(M_p\cong\CP^{n-1}\) be the projectivization of the span of $\{(a_0,a_1,0,\ldots,0),e_2,\ldots,e_n\}.$
Then \(M_p\) is totally geodesic. Let $\nu$ be the unit normal to the $U(n-1)$-orbit $\mathcal{S}$ through $p$
inside the totally geodesic $\CP^{n-1}$.
Since $v$ is constant along the $U(n-1)$-orbit,
\(
H(J\nu,J\nu)=-dv(\mathds I_{\mathcal{S}}(J\nu,J\nu))
\)
by (\ref{vertical Hessian formula}).
By the $U(n-2)$-isotropy,
there exist $\kappa,\eta\in\mathbb R$ such that
\(
\mathds I_{\mathcal{S}}(W,W)=-\kappa\,\nu,
\) 
for every unit $W\in N_p\Sigma$, and
\(
\mathds I_{\mathcal{S}}(J\nu,J\nu)=-\eta\,\nu
\) with the mixed components vanishing. By Codazzi equation, we have
\[
2
=
\operatorname{Rm}(W,JW,J\nu,\nu)
=
\left\langle
(\nabla_WA)(JW)-(\nabla_{JW}A)W,J\nu
\right\rangle,
\]
where $A$ is the shape operator for $\mathcal{S}$. 

To evaluate the Codazzi equation, extend \(W\) locally as a unit
section of the \(-\kappa\)-principal distribution; since this
distribution is \(J\)-invariant, \(JW\) is another such section.
Thus
\[
AW=-\kappa W,\qquad A(JW)=-\kappa JW,
\qquad A(J\nu)=-\eta J\nu
\]
along the extension.  Using \(\overline\nabla J=0\) and
\(\overline\nabla_U\nu=-AU\), we have
\[
\begin{aligned}
\langle\nabla_WJW,J\nu\rangle
&=\langle\overline\nabla_WJW,J\nu\rangle
 =\langle\overline\nabla_WW,\nu\rangle
 =-\kappa,\\
\langle\nabla_{JW}W,J\nu\rangle
&=-\langle W,\overline\nabla_{JW}J\nu\rangle
 =-\langle\overline\nabla_WW,\nu\rangle=\kappa.
\end{aligned}
\]
Consequently,
\[
\begin{aligned}
\left\langle(\nabla_WA)(JW),J\nu\right\rangle
&=(\eta-\kappa)
  \langle\nabla_WJW,J\nu\rangle
 =\kappa(\kappa-\eta),\\
\left\langle(\nabla_{JW}A)W,J\nu\right\rangle
&=(\eta-\kappa)
  \langle\nabla_{JW}W,J\nu\rangle
 =-\kappa(\kappa-\eta).
\end{aligned}
\]
Hence
\(
\left\langle
(\nabla_WA)(JW)-(\nabla_{JW}A)W,J\nu
\right\rangle
=
2\kappa(\kappa-\eta),
\)
which implies $\kappa(\kappa-\eta)=1$. Using $-c=H(W,W)=\kappa dv(\nu)$, we have
\(
H(J\nu,J\nu)=\eta\,dv(\nu)
=-\frac{c\eta}{\kappa},
\)
and hence
\[
c+H(J\nu,J\nu)
=
c\left(1-\frac{\eta}{\kappa}\right)
=
\frac{c}{\kappa^2},
\]
which proves $dv(\nu)^2=c(c+H(J\nu,J\nu))$ and finishes the proof.
\end{proof}

By combining the above estimate, we can rule out the null contact on the principal orbit stratum:
\begin{proposition}[Exclusion of principal null contacts in
\(\CP^3\) and \(\CP^4\)]
\label{prop:principal-exclusion-CP34}
Let \(n\in\{3,4\}\), and let \(u>0\) be a
\(U(1)\times U(n-1)\)-invariant solution of
\(-\Delta u=\lambda u\), with \(\lambda>0\), on an invariant open
subset of \(\CP^n\).  If \(H\leq0\) near a point \(p\) on the
principal stratum, then \(H_p<0\).
\end{proposition}

\begin{proof}By \eqref{eq:reduced-maximal-eigenvalue}, it remains to show that
\[
\frac{12dv(JX)^2\bigl(H(Y,Y)+c\bigr)}
{|\nabla v|^2}
<
\frac{4|\nabla v|^2
\bigl(|\nabla v|^2+3dv(JX)^2\bigr)}
{dv(JY)^2\,(-H(Y,Y))}.
\]
If \(H(Y,Y)+c\leq0\), this is immediate.  Suppose therefore that
\(H(Y,Y)+c>0\).  By
Lemma~\ref{lem:principal-gradient-estimate}, we have
\[
\frac{12dv(JX)^2\bigl(H(Y,Y)+c\bigr)}
{|\nabla v|^2}
\leq
\frac{12dv(JX)^2}{-H(Y,Y)}.
\]
On the other hand, since
\(
|\nabla v|^2=dv(JX)^2+dv(JY)^2,
\)
we have
\[
|\nabla v|^2\geq dv(JY)^2,
\qquad
|\nabla v|^2+3dv(JX)^2>3dv(JX)^2.
\]
Therefore
\[
\frac{4|\nabla v|^2
\bigl(|\nabla v|^2+3dv(JX)^2\bigr)}
{dv(JY)^2\,(-H(Y,Y))}
>
\frac{12dv(JX)^2}{-H(Y,Y)}.
\]
Combining the last two inequalities gives the desired estimate, which
contradicting the fact that \(\mu_1\) attains a local maximum at
\(p\).  Hence no null contact can occur, and \(H_p<0\).
\end{proof}

Since all null-contacts have been excluded in our setting, we can prove the main theorem now:
\begin{proof}[Proof of the main theorem]
If strict log-concavity failed, Proposition~\ref{reduction to null contact} would produce an invariant weakly log-concave first eigenfunction with an interior null contact.
Such a contact is excluded on the singular stratum by Proposition~\ref{lem:singular-orbit-exclusion}, and on the principal stratum by Proposition~\ref{prop:principal-exclusion-CP2} when $n=2$ and Proposition~\ref{prop:principal-exclusion-CP34} when $n=3,4$.
Thus no null contact exists, and the first Dirichlet eigenfunction is strictly log-concave.
\end{proof}

\printbibliography

\end{document}